\documentclass{article}
\usepackage{amsmath}
\usepackage{amsthm}
\usepackage{amsfonts}
\usepackage{amssymb}
\usepackage[all]{xy}
\usepackage{amscd}

    \newtheorem{Lem}{Lemma}[section]
    \newtheorem{Lem-Def}{Lemma-Definition}[section]
    \newtheorem{Prop}[Lem]{Proposition}

    \newtheorem{Thm}[Lem]{Theorem}

\theoremstyle{definition}

    \newtheorem{Def}[Lem]{Definition}
    \newtheorem{Exa}[Lem]{Example}
    \newtheorem{Rem}[Lem]{Remark}

\DeclareMathOperator{\Pic}{Pic}

\newcommand{\W}{\mathcal W}

\newcommand{\I}{\mathcal I}

\newcommand{\N}{\mathcal N}
\newcommand{\T}{\mathcal T}

\renewcommand{\L}{\mathcal L}
\renewcommand{\O}{\mathcal O}
\newcommand{\C}{\mathcal C}

\newcommand{\Y}{\mathcal Y}
\newcommand{\col}{\colon}

\newcommand{\ra}{\rightarrow}
\newcommand{\ol}{\overline}

\newcommand{\ot}{\otimes}
\newcommand{\Q}{\mathbb{Q}}
\newcommand{\ze}{\mathbb{Z}}

\newcommand{\ul}{\underline}

\newcommand{\tx}{\text}

\newcommand{\lra}{\longrightarrow}
\renewcommand{\:}{\colon}

\begin{document}

\title{Abel maps for curves of compact type}

\author{
Juliana Coelho and Marco Pacini 
\footnote{
Corresponding author: tel.(55)(21)2629-2078, fax.(55)(21)2629-2113 
\newline E-mail addresses: coelho@impa.br, pacini@impa.br
\newline Second author supported by Faperj, Prog. INST, Proc. E-26/110.774/2008}
\\ Universidade Federal Fluminense
\\ Rua M\'ario Santos Braga S/N 
\\ Niter\'oi--Rio de Janeiro--Brazil
}

\maketitle

\begin{abstract}
\noindent
Recently, the first Abel map for a stable curve of genus 
$g\ge 2$ has been constructed. Fix an integer $d\ge 1$ and
let $C$ be a stable curve of compact type of genus $g\ge 2$. 
We construct two $d$-th Abel maps for $C$, having different targets, and we compare the fibers of the two maps. As an application, we get a characterization of hyperelliptic stable curves of compact type with two components via the $2$-nd Abel map.
\end{abstract}



\section{Introduction}

\subsection{Abel maps of singular curves}

The subject of this paper is the completion of Abel maps for singular curves. If $C$ is a smooth projective curve, the $d$-th Abel map is the geometrically meaningful morphism: 

$$
\begin{array}{rcl}
C^d &\lra& \Pic^d C\\
(p_1,\dots,p_d)&\mapsto& \displaystyle \O_C(\sum_{1\le i\le d} p_i)
\end{array}
$$
where $C^d$ is the product of $d$ copies of $C$ and $\Pic^d C$ is the degree-$d$ Picard variety of $C$. It makes sense, the problem of defining an analogous map when $C$ is singular, preserving a geometrical meaning.
This problem has been considered by several authors in the last three decades. It has been completely solved for integral curves in \cite{AK} and in degree one in \cite{CE} and \cite{CCE}, but a general analysis is still missing. The purpose of this paper is  to consider the special case of Abel maps for curves of compact type, i.e. nodal curves with only separating nodes. This assumption allows us to avoid several combinatorial problems that naturally arise when one considers the general case.

The Abel map of a smooth curve has the remarkable property that its fibers are projectivized complete linear series (up to the action of the $d$-th symmetric group). It turns out that an important motivation for studying Abel maps is the attempt of giving a new definition of limit linear series on a nodal curve. In the case of curves of compact type, it would be interesting to establish the relationship with the definition of limit linear series introduced in \cite{EH} or \cite{O}.

\subsection{The main result}

Fix a nodal curve $C$ over an algebraically closed field. Let $\gamma$ be the number of irreducible components of $C$. 
Let $f\:\C\ra B$ be a smoothing of $C$, i.e. a family of curves such that $\C$ is  smooth and $B$ is smooth, one-dimensional, with a distinguished point $0\in B$ such that $f^{-1}(0)=C$ and $f^{-1}(b)$ is smooth if $b\ne 0$. 

Fix an integer $d\ge 1$ and let $J^d_f$ be the degree-$d$ relative generalized Jacobian of the family. Recall that in general $J^d_f$ can be constructed just as an algebraic space. Let $\C^d$ be the product of  $d$ copies of $\C$ over $B$ and consider the relative rational map: 
$$
\begin{array}{rcl}
\C^d & \stackrel{\alpha^d_f}{\dashrightarrow} & J^d_f\\
(p_1,\dots,p_d)&\mapsto& \displaystyle \O_{f^{-1}(f(p_1))}(\sum_{1\le i\le d} p_i)
\end{array}
$$
which is the relative $d$-th Abel map away from the fiber over the distinguished point $0\in B$. 
If $C$ is of compact type, we will construct a morphism:

\begin{equation}\label{ext}
\ol{\alpha^d_f}\:\C^d\lra J^d_f
\end{equation}
extending $\alpha^d_f$. 
Let $J^d_C$ be the degree-$d$ generalized Jacobian of $C$. The fact that $\ol{\alpha^d_f}$ can be constructed with $J^d_f$ as target, implies the existence of a unique $\gamma$-tuple $\ul{e}_d\in\ze^\gamma$ such that that $\ol{\alpha^d_f}|_{C^d}$ factors through $J^{\ul{e}_d}_C\subset J^d_C$, where  $J^{\ul{e}_d}_C$ is the locus parametrizing line bundles on $C$ whose degrees on the irreducible components are prescribed by $\ul{e}_d$. The main result of this paper is to show that we can construct a canonical extension as in (\ref{ext}) such that the associated $\gamma$-tuple $\ul{e}_d$ has geometrically meaningful properties.

More precisely, for every $\gamma$-tuple $\ul{a}=(a_1,\dots,a_\gamma)\in\Q^\gamma$ such that $a_i\ge 0$, for $i=1,\dots,\gamma$, and $a_1+\dots+a_\gamma=1$, the definition of \emph{$\ul{a}$-semi-stable} and \emph{$\ul{a}$-stable} torsion-free, rank-1 sheaves on $C$ is introduced in \cite{S}. These are sheaves satisfying certain numerical conditions involving the numbers $a_1,\dots, a_\gamma$. 
An intermediate notion of stability is introduced in \cite{E01} by means of $X$-quasistability, for every irreducible component $X$ of $C$. There it is constructed a proper scheme $J^{d,X}_C$ parametrizing $X$-quasistable rank-1 torsion-free sheaves of $C$ of degree $d$.
If $L$ is a $X$-quasistable line bundle of $C$ and if $\ul{d}_L$ is the $\gamma$-tuple whose entries are the degrees of $L$ on the irreducible components of $C$, we say that  $\ul{d}_L$  is a \emph{$X$-quasistable multidegree}.

 We sum up our main result in Theorem \ref{Main}, relating extensions of Abel maps for smoothing of curves of compact type and $X$-quasistable multidegrees.

\begin{Thm}\label{Main}
Fix integers $g\ge 2$ and $d\ge 1$. Let $f\:\C\ra B$ be a smoothing of a stable curve $C$ of compact type of genus $g$. Then there exist a distinguished component $X^{pr}$ of $C$, a $X^{pr}$-quasistable multidegree $\ul{e}_d$ and a morphism: 
$$\ol{\alpha^d_f}\:\C^d \lra J^d_f$$ 
as in (\ref{ext}) such that the following properties hold:
\begin{itemize}
\item[(i)]
$\ol{\alpha^d_f}|_{C^d}$ does not depend on the choice of $f$, and factors through the immersion $J^{\ul{e}_d}_C\hookrightarrow J^d_C$.
\item[(ii)]
if $S^d(C)$ is the $d$-th symmetric product of $C$, then $\ol{\alpha^d_f}|_{C^d}$  factors via a morphism 
$\beta^d_C\:S^d(C)\ra J_C^{\ul{e}_d}$ which does not depend on the choice of  $f$.
\end{itemize}
\end{Thm}

We point out that the determination of the component $X^{pr}$ and the $X^{pr}$-quasistable multidegree $\ul{e}_d$ is effective, as it is rather clear from the proof of Lemma \ref{centracomp}, and from the iterative procedure yielding Definition \ref{Main-Def}.  

The property (i) of Theorem \ref{Main} allows us to consider another $d$-th Abel map for $C$ whose target is the compactification of the universal Picard variety constructed in \cite{C}. In Proposition \ref{same-fiber}, we will see that the set-theoretic fibers of the two Abel maps are equal. Recall that the same phenomenon takes place for the first Abel maps of a stable curve (see \cite{CE} and \cite{CCE}).
However, in Remark \ref{higher-maps}, we will produce an example, hinting 
that it should not hold  for higher Abel maps of curves not of compact type. 

\smallskip

As an application of Theorem \ref{Main} (ii), we give the following characterization of hyperelliptic curves of compact type with two components.

\begin{Prop}\label{Application}
Fix an integer $g\ge 2$, and let $C$ be a stable curve of compact type of genus $g$ with two components. Then $C$ is hyperelliptic if and only if there exists a fiber of the morphism $\beta^2_C$ consisting of two smooth rational curves intersecting at one point.
\end{Prop}

We plan to investigate whether a similar characterization could be given for hyperelliptic curves with many components or for trigonal curves.

\subsection{Notation and Terminology}

A \emph{curve} is a connected, projective and reduced scheme of dimension 1 over an algebraically closed field $k$. If $C$ is a curve, then $g_C:=1-\chi(\O_C)$ is the \emph{genus} of $C$ and $\omega_C$ is its dualizing sheaf. We will always consider curves with  nodal singularities.

A \emph{subcurve} of $C$ is a union of irreducible components of $C$. If $Y$ is a proper subcurve of $C$, we let $Y':=\overline{C-Y}$ and call it the \emph{complementary subcurve} of $Y$. We denote $k_Y:=\#(Y\cap Y')$.  
A subcurve $Y$ of $C$ is said to be a \emph{tail} of $C$ if $k_Y=1$. 
In this case, the intersection $Y\cap Y'$ consists of a single node $n$ and we say that $n$ is a \emph{separating node} of $C$. We remark that a separating node defines two tails $Y$ and $Y'$ on $C$ such that $Y\cap Y'=\{n\}$.

A \emph{stable} curve $C$ is a nodal curve such that every smooth rational subcurve of $C$ meets the rest of the curve in at least $3$ points. 
A curve $C$ is said to be \emph{of compact type} if its only singularities are separating nodes. 
A curve $W$ is obtained by \emph{blowing up} a curve $C$ at a subset $\Sigma$ of its nodes, if there is a morphism $\pi\col W\ra C$ such that, for every $p\in\Sigma,$ we have $\pi^{-1}(p)\simeq\mathbb{P}^1$ and $\pi\col W-\cup_{p\in \Sigma} \pi^{-1}(p)\ra C-\Sigma$ is an isomorphism.

A \emph{family of curves} is a proper and flat morphism $f\col\mathcal C\ra B$ whose fibers are curves. If $b\in B$, we denote $C_b:=f^{-1}(b)$. 
A \emph{smoothing} of a curve $C$ is a family $f\col\C\ra B,$ where $\C$ is smooth and $B$ is a smooth curve with a distinguished point $0\in B$ such that $C_b$ is smooth for $b\ne 0$ and $C_0=C$.

If $f\:\C\ra B$ is a family of curves $C$, we denote by $\C^d$ the product of $d$ copies of $\C$ over $B$ and by $S^d(\C)$ the $d$-th symmetric product of $C$ over $B$, i.e. the quotient of  $\C^d$ by the action of the $d$-th symmetric group.

 The \emph{degree} of a line bundle $L$ on a curve $C$ is $\deg(L):=\chi(L)-\chi(\O_C)$.

\section{Technical background}

\subsection{Jacobians of singular curves}

If not otherwise specified, in this section $C$ will be a nodal curve with irreducible components $X_1,\ldots,X_{\gamma}$. Let $J_C^d$ be the degree-$d$ generalized Jacobian of $C$, parametrizing line bundles of degree $d$ on $C$. Since $C$ is a proper scheme, $J_C^d$ is a scheme (see \cite[Theorem 8.2.3]{BLR}). We have the following decomposition of $J^d_C$:
\begin{equation}\label{Jacobdecomp}
J^d_C=\underset{ d_1+\ldots+d_{\gamma}=d}{\underset{\ul{d}=(d_1,\dots,d_\gamma)}{\cup}}J^{\ul{d}}_C,
\end{equation}
where $J^{\ul{d}}_C$ is a connected component of $J_C^d$ parametrizing line bundles $L$ such that $\deg(L|_{X_i})=d_i$  for $i=1,\dots,\gamma$. 
If $C$ is of compact type, then 
for each $\ul{d}=(d_1,\ldots,d_{\gamma})$, we have an isomorphism:
\begin{equation}\label{iso-compo}
\begin{array}{rcl}
J^{\ul{d}}_C &\stackrel{\sim}{\lra}& J_{X_1}^{d_1}\times\cdots\times J_{X_{\gamma}}^{d_{\gamma}}\\
L&\mapsto& \displaystyle (L|_{X_1},\dots, L|_{X_\gamma})
\end{array}.
\end{equation}

Let  $\ul{d}=(d_1,\ldots,d_{\gamma})\in\ze^d$.
If $L\in J^{\ul{d}}_C$, we say that $L$ has \emph{multidegree} $\ul{d}$ and that $d_1+\ldots+d_{\gamma}$ is the \emph{total degree} of $L$.
 For each subcurve $Y$ of $C$, set: 
$$d_Y:=\sum_{X_i\subseteq Y}d_i.$$

Fix an integer $d$. A \emph{polarization} on $C$ is a vector bundle $E$ on $C$
of rank $r>0$ and relative degree $r(g_C-1-d)$. We will denote by
$E_d$ the (canonical) polarization on $C$:
\begin{equation}\label{polar}
E_d=
\begin{cases}
\begin{array}{ll}
\omega_C^{\otimes (g_C-1-d)}\oplus \mathcal O_C^{\oplus (2g_C-3)} & d\ne g_C-1 \\
\mathcal O_C & d=g_C-1
\end{array}.
\end{cases}
\end{equation}

Let $L$ be a line bundle on $C$ with multidegree $\ul{d}$. 
We say that $L$, or $\ul{d}$, is (canonically) \emph{semistable} if for every non-empty, proper subcurve $Y\subsetneq C$ we have: 
$$
\chi(L|_Y) \ge \frac{-\deg E_d|_Y}{\text{rank}(E_d)}.
$$
Moreover, if $X$ is any component of $C$, we say that $L$, or $\ul{d}$,  is (canonically) \emph{$X$-quasistable} if it is semistable and: 
$$
\chi(L|_Y)>\frac{-\deg E_d|_Y}{\text{rank}(E_d)},
$$
whenever $Y\supseteq X$.  It is not difficult to prove  that a degree-$d$ line bundle $L$ on $C$ is semistable with respect to $E_d$ if and only if for every non-empty, proper subcurve 
$Y\subsetneq C$:
\begin{equation}\label{BI1}
\left|\deg L|_Y -\frac{d}{2g-2} \deg \omega_C|_Y \right|\le \frac{k_Y}{2}.
\end{equation}
 Moreover, $L$ is $X$-quasistable with respect to $E_d$ if and only if (\ref{BI1}) is satisfied and: 
\begin{equation}\label{BI2}
\deg L|_Y -\frac{d}{2g-2} \deg \omega_C|_Y > - \frac{k_Y}{2},
\end{equation}
whenever $Y\supseteq X$ 
(see \cite[Lemma 3.1]{P}). If the condition (\ref{BI1}) (respectively  (\ref{BI2})) is satisfied for a degree-$d$ line bundle $L$ on $C$ of multidegree $\ul{d}$ and a subcurve $Y$ of $C$ (respectively a subcurve $Y$ of $C$ such that $X\subseteq Y$), we say that $L$, or $\ul{d}$, is \emph{semistable} (respectively \emph{$X$-quasistable}) at $Y$. It is easy to see that $\ul{d}$ is semistable at $Y$ if and only if it is semistable at $Y'$.
We define subschemes of $J_C^d$ by: 
$$J^{d,ss}_C=\underset{d_1+\ldots+d_{\gamma}=d}{\underset{\ul{d}\tx{ is semistable}}{\cup}}J^{\ul{d}}_C
 \quad\tx{and} \quad
J^{d,X}_C=\underset{d_1+\ldots+d_{\gamma}=d}{\underset{\ul{d}\tx{ is $X$-quasistable}}{\cup}}J^{\ul{d}}_C.$$

Now, let $f\:\C\to B$ be a family of nodal curves $C$. We denote by  
$J^d_f$ the relative degree-$d$ generalized Jacobian of the family $f$.
In general, $J^d_f$ can be constructed just as an algebraic space. 
Nevertheless, there exists an \'etale base change $B'\ra B$ such that, if we consider the pull-back family:
$$
\Y:=\C\times_B B'\stackrel{f'}{\lra} B' ,
$$ 
then $J^d_{f'}$ is a $B'$-scheme and there exists a universal line bundle over $J^d_{f'}\times_{B'}\Y$.

\subsection{The first Abel map}

In \cite{CE} and \cite{CCE}, the authors constructed the first Abel map for a smoothing of a stable curve. More precisely, fix a smoothing $f\:\C\ra B$ of a stable curve $C$. For our purposes, we may assume that $C$ is of compact type. Let $J_f^1$ be the degree-1 relative generalized Jacobian of $f$. Then there exists a morphism:
\begin{equation}\label{1stabelmapED}
\ol{\alpha^1_f}\: \C\lra  J^1_f.
\end{equation}
extending the relative first Abel map of the family of smooth curves $f|_{f^{-1}(B-0)}$. We will  recall the definition of $\ol{\alpha^1_f}|_C$ in (\ref{first-def}).

For every $g\ge 2$, let $\ol{M_g}$ be the moduli space of Deligne--Mumford stable curves of genus $g$. If $C$ is a stable curve, $[C]$ will denote the point of  $\ol{M_g}$ parametrizing $C$. If $g$ is even, let $\Delta_{g/2}$ be the divisor of $\ol{M_g}$ which is the closure of the locus parametrizing curves $C=X_1\cup X_2$ such that  $g_{X_1}=g_{X_2}=g/2$ and $\#(X_1\cap X_2)=1$.

\begin{Def}\label{map-def}
An irreducible component $X$ of $C$ is \emph{central} 
(respectively \emph{semicentral}) if $g_Z<g_C/2$ (respectively $g_Z\le g_C/2$) for every connected component $Z$ of $X'$.
\end{Def}

\begin{Lem}\label{centracomp}
Fix an integer $g\ge 2$. Let $C$ be a stable curve of compact type of genus $g$. Let $M_C$ (respectively $N_C$) be the number of central (respectively semicentral) components of $C$. Then the following properties hold:
\begin{itemize}
\item[(i)]
$M_C=1$ if and only if $[C]\notin\Delta_{g/2}$.
\item[(ii)]
$M_C=0$ if and only if $[C]\in\Delta_{g/2}$.
\item[(iii)]
if $M_C=0$, then $N_C=2$ and the intersection of the two semicentral components is non-empty. 
\end{itemize}
\end{Lem}

\begin{proof}

First of all, notice that $[C]\in\Delta_{g/2}$ if and only if there exists a node which is the intersection of two tails of $C$ of genus $g/2$.

\smallskip

\emph{First step}. We claim that $M_C\le 1$. Indeed, suppose that $X_1$ and $X_2$ are central components. Let $Z_1$ and $Z_2$ be the connected components of $X'_1$ and $X'_2$ containing respectively $X_2$ and $X_1$.  Then $g_{Z_i}<g/2$ for $i=1,2$, because $X_1$ and $X_2$ are central. Set $Y_1:=Z_1\cap Z_2$ and $Y_2:=\ol{X'_1-Z_1}$.  Since $C$ is of compact type, we have  $Z_2=Y_1\cup Y_2\cup X_1$. In particular, if $p$ is a point of $C$ with $p\notin Z_1$, then $p\in X_1\cup Y_2\subseteq Z_2$. Hence $C=Z_1\cup Z_2$. Since the genus of a curve of compact type is the sum of the genus of its irreducible components, we have $g\le g_{Z_1}+g_{Z_2}<g$ yielding a contradiction.

\smallskip

\emph{Second step}. For any irreducible component $X$ of $C$, we define an irreducible component $Y(X)$ of $C$ and a connected component $Z(X)$ of $X'$ as follows. If $X$ is central, set $Y(X):=X$ and we let $Z(X)$ be any connected component of $X'$. If $X$ is not central, we claim that there exists exactly one connected component $Z(X)$ of $X'$ such that $g_{Z(X)}\ge g/2$. Indeed, this is clear if $g_X>0$. If $g_X=0$, then $X'$ has at least 3 connected components of genus at least 1, because $C$ is stable,  hence there exists a unique connected component $Z(X)$ of $X'$ with $g_{Z(X)}\ge g/2$. 
Denote by $Y(X)$ the irreducible component of $Z(X)$ intersecting $X$. 
Notice that $Y(X)=X$ if and only if $X$ is central.

\smallskip

\emph{Third step}. We claim that if $X$ and $Y(X)$ are not central, and if $[C]\notin\Delta_{g/2}$, then:
\begin{equation}\label{WYX}
Z(Y(X))\subseteq Z(X)-Y(X).
\end{equation}

Indeed, assume that $X$ is not central. Now, $Z(X)'$ is a connected component of $Y(X)'$.  Since $g_{Z(X)}\ge g/2$ and $[C]\notin \Delta_{g/2}$, we have 
$g_{Z(X)'}<g/2$. Recall that, since $Y(X)$ is not central, $Z(Y(X))$ is the connected component of  $Y(X)'$ such that $g_{Z(Y(X))}\ge g/2$. Then $Z(Y(X))\ne Z(X)'$, and $Z(Y(X))\subseteq Z(X)-Y(X)$.

\smallskip

\emph{Fourth step}. We show (i) and (ii). First of all, we prove that if $[C]\notin\Delta_{g/2}$, then $M_C\ge 1$.
 Let $X_1$ be any irreducible component of $C$. Set $X_2=Y(X_1)$ and by induction $X_i=Y(X_{i-1})$ for every $i>2$. If $X_r$ is central, for some positive integer $r$, then we are done. Otherwise, $X_r\ne X_{r+1}$ for every $r\ge 1$ and by 
(\ref{WYX}), we get an infinite chain of subcurves:
$$\dots\subsetneq Z(X_r)\subsetneq \dots \subsetneq  Z(X_2)\subsetneq Z(X_1)$$  
yielding a contradiction. 
Then $M_C\ge 1$ and hence $M_C=1$ by the first step.

Conversely, we prove that if $M_C=1$, then $[C]\notin\Delta_{g/2}$. Let $X$ be the central component. Then every node of $C$ is the intersection of two tails 
$W_1$ and $W_2$ such that $W_1\subseteq X'$ and $X\subseteq W_2$. Thus 
$g_{W_1}<g/2$, and $[C]\notin\Delta_{g/2}$.

Of course, the first step and (i) imply (ii). 

\smallskip

\emph{Fifth step}. We show (iii). Let $M_C=0$, i.e. $[C]\in \Delta_{g/2}$. First of all, we show that $N_C\ge 2$. Since $[C]\in \Delta_{g/2}$, there exist tails $Z_1$ and $Z_2$ of genus $g/2$ intersecting in a node. Assume that $X_1$ and $X_2$ are the irreducible components such that $X_1\subseteq Z_1$, $X_2\subset Z_2$, $X_1\cap X_2\ne\emptyset$. Then $X_1$ and $X_2$ are semicentral components 
intersecting in a node.

We show that $N_C\le 2$. Indeed, suppose that $X_1$, $X_2$, $X_3$ are semicentral components. Notice that $X_1$, $X_2$, $X_3$ are not central, because $M_C=0$. Then 
there exist at least 4
distinct tails $Z_1,\dots, Z_4$ of genus $g/2$. Since $C$ is of compact type, up to change the index of the tails, we may assume that $Z_1\subseteq Z_3$ and $Z_4\subseteq Z_2$. Thus $(Z_1\cup Z_4)'\ne\emptyset$, because the tails $Z_1,\dots,Z_4$ are distinct. Hence $g_{(Z_1\cup Z_4)'}=g-g_{Z_1}-g_{Z_2}=0$ 
and $k_{(Z_1\cup Z_4)'}=2$. This is a contradiction because, since $C$ is stable, $k_Y\ge 3$ for every non-empty subcurve $Y\subsetneq C$ with $g_Y=0$.
\end{proof}

\begin{Def} 
Fix an integer $g\ge 2$. Let $C$ be a stable curve of compact type of genus $g$. 
If $[C]\notin\Delta_{\frac{g}{2}}$, let $X^{pr}$ be the central component of $C$. If $[C]\in\Delta_{\frac{g}{2}}$, let $X^{pr}$ be one of the two semicentral components of $C$.
We will keep this choice throughout the paper. 
We call $X^{pr}$ \emph{the principal component of} $C$.
A tail $Z$ of $C$ is \emph{small} if at least one of the following conditions is satisfied:
\begin{itemize}
\item[(1)] 
$g_Z<g_C/2$ 
\item[(2)]
$g_Z=g_C/2$ and $Z'\supseteq X^{pr}$.
\end{itemize} 
\end{Def}

Fix an integer $g\ge 2$. Let $C$ be a stable curve of compact type of genus $g$. 
We recall now some properties of the first Abel map (\ref{1stabelmapED}).
We define for each $q\in C$ a line bundle $\N_q$ on $C$ as follows. If $q$ is a smooth point of $C$, then $\N_q:=\mathcal O_C(q)$. 
If $q$ is a node of $C$, let $Z$ be the small tail attached to $q$. Using the isomorphism (\ref{iso-compo}), there exists a unique line bundle $\N_q$ on $C$ such that $\N_q|_Z=\O_Z(q)$ and $\N_q|_{Z'}=\O_{Z'}$. Then 
$\ol{\alpha^1_f}|_C$ sends each $q\in C$ to:
\begin{equation}\label{first-def}
[\N_q\otimes\O_{\C}(\underset{q\in Z}{\underset{Z\text{ small tail}}{\sum}} Z)|_C)]\in J_C^1
\end{equation}

The morphism $\ol{\alpha^1_f}|_C$ does not depend on the choice of $f$, and factors through the immersion $J^{1,ss}_C\hookrightarrow J_C^1$.

\begin{Prop}\label{firstdeg}
Fix an integer $g\ge 2$. Let $C$ be a stable curve of compact type of genus $g$. Let $X^{pr}$ be the principal component of $C$. 
Let $\ul{e}_1$ be the multidegree such that $(e_1)_{X^{pr}}=1$ and $(e_1)_X=0$ for every irreducible component $X\subseteq C$ such that  
$X\neq X^{pr}$. 
Then the following properties hold:
\begin{itemize}
 \item[(i)] 
$J^{1,X^{pr}}_C=J^{\ul{e}_1}_C$;
\item[(ii)]
$\ol{\alpha^1_f}|_C$ factors through the immersion $J^{\ul{e}_1}_C \hookrightarrow J^{1,ss}_C$.
\end{itemize}
\end{Prop}

\begin{proof} 
We claim that if $\ul{d}$ is a $X^{pr}$-quasistable multidegree of total degree 1, then $\ul{d}=\ul{e}_1$. We show this claim in 3 steps.

\smallskip

\emph{First step}. We show that if $Z$ is a tail of $C$, then $d_Z=1$ if $Z$ contains $X^{pr}$, and $d_Z=0$ otherwise. It 
 suffices to show that $d_Y=0$ for each tail $Y$ of $C$ such that that $X^{pr}\not\subseteq Y$. If $Y$ is such a tail, then there exists a connected component $W$ of $(X^{pr})'$ such that $Y\subset W$ and 
hence $g_Y\le g_W\le g/2$. Being $k_Y=1$, from (\ref{BI1}) and (\ref{BI2}) with $d=1$, we have:
$$-\frac12\leq d_Y- \frac{\deg(\omega_C|_Y)}{2g-2}< \frac{1}{2}.$$  Since $g_Y\le g/2$, we get:
$$0\leq\frac{\deg(\omega_C|_Y)}{2g-2}=\frac{2g_Y-1}{2g-2}\leq\frac12.$$
So $-1/2\leq d_Y<1$, and hence $d_Y=0$.

\smallskip

\emph{Second step}.
We show that $d_{X^{pr}}=1$. Let $Z_1,\ldots,Z_n$ be the connected components of $(X^{pr})'$. Then $g_{Z_i}\le g/2$ for every $i=1,\dots,n$ and by the first step we have $d_{Z_i}=0$ for every $i= 1\dots,n$. Now, $Z_1'$ is a tail of $C$ containing $X^{pr}$. Hence, again by the first step, $d_{Z_1'}=1$. Moreover, since $X^{pr}=\overline{Z_1'-(Z_2\cup\ldots\cup Z_n)}$, we get: 
$$d_{X^{pr}}=d_{Z_1'}-\sum_{2\leq i\leq n}d_{Z_i}=1.$$

\emph{Third step}. We show that if $\ul{d}$ is a $X^{pr}$-quasistable multidegree of total degree 1, then $\ul{d}=\ul{e}_1$.  
By the second step, we are done if we show that $d_X=0$ for every irreducible component $X$ of $C$ such that $X\ne X^{pr}$. Indeed, let $X$ be any such component.
Set $C_0:=C$. For every $j\ge 1$, define inductively $C_j:=\ol{C_{j-1}-Y_j}$, where $Y_j$ is the union of the irreducible components of $C_{j-1}$ which are tails of $C_{j-1}$ distinct from $X^{pr}$. Notice that $C_j$ is a curve of compact type, then if $C_j\ne X^{pr}$, then there exist at least two irreducible component of $C_j$ which are tails of $C_j$. In this way, if $C_j\ne X^{pr}$, then $Y_j\ne \emptyset$ and $C_{j+1}\subsetneq C_j$. Since $C$ has a finite number of components, there exists an integer $r\ge 0$ such that $X$ is a tail of $C_r$. Then there are irreducible components $X_1,\dots X_n$ of $C$ contained in $Y_1\cup Y_2\cup \dots \cup Y_r$ such that $X\cup X_1\cup\cdots\cup X_n$ is a tail of $C$ not containing $X^{pr}$ and $X_1\cup X_2\cup\cdots \cup X_n$ is a union of  tails of $C$ not containing $X^{pr}$. By the first step, $d_{X\cup X_1\cup\cdots \cup X_n}=d_{X_1\cup X_2\cdots \cup X_n}=0$, and hence:
$$d_X=d_{X\cup X_1\cup\dots\cup X_n}-d_{X_1\cup  X_2\cdots \cup X_n}=0.$$

In this way, we have shown the initial claim.

\smallskip

By \cite[Theorem 5.5]{CCE}, for every irreducible component $Y\subset C$ not contained in any small tail, we have that  $\ol{\alpha^1_f}|_C$ factors through the immersion $J^{1,Y}_C\hookrightarrow J^{1,ss}_C.$ By definition, $X^{pr}$ is not contained in any small tail of $C$, hence $\ol{\alpha^1_f}|_C$ factors through the immersion $J^{1,X^{pr}}_C\hookrightarrow J^{1,ss}_C$. This implies that $J^{1,X^{pr}}_C\ne\emptyset$. By the initial claim we have $J^{1,X^{pr}}_C=J^{\ul{e}_1}_C$ and we are done.
\end{proof}

\section{Abel maps for curves of compact type}

\subsection{The construction of the $d$-th Abel map}

In this section we will construct the Abel map for a smoothing of a stable curve of compact type.

\begin{Def}
Let $C$ be a stable curve of compact type with $\gamma$ components. 
Fix a multidegree $\ul{d}=(d_1,\dots, d_\gamma)$ of total degree $d$.
We say that a tail $Z$ of $C$ is a 
\emph{$\ul{d}$-big tail} if: 
$$d_Z\cdot \deg(\omega_C)-d \cdot \deg(\omega_C|_Z)<2g_Z-g_C.$$
Notice that if $\ul{d}=(0,0,\dots,0)$, then 
the notion of $\ul{d}$-big tail coincides with the notion of big tail in 
\cite{CE}. For each irreducible component $X$ of $C$ and for each multidegree 
$\ul{d}$, define:
$$\T_{\ul{d}}(X):=\{Z\subset C : Z \tx{ is a $\ul{d}$-big tail of } C \tx{ and } Z\nsupseteq X\}.$$
\end{Def}

Let $f\:\C\ra B$ be a smoothing of a stable curve $C$ of compact type.
For each $d\ge 1$, let $B_d\ra B$ be an \'etale morphism such that, if we consider the pull-back family:
$$
\Y_d:=\C\times_{B_d}B\stackrel{f_d}{\lra}B_d,
$$ 
then $J^1_{f_d}$ and $J^d_{f_d}$ are schemes and there exists a universal line bundle $\L_1$ (respectively $\L_d$) on $J^1_{f_d}\times_{B_d}\Y_d$ 
(respectively $J^d_{f_d}\times_{B_d} \Y_d$). 
We have a natural \'etale morphism $J^d_{f_d}\ra J^d_f$. For each multidegree $\ul d$ of total degree $d$, 
consider the  line bundle on $J^1_{f_d}\times_{B_d} \Y_d$:
\begin{equation}\label{modif}
T_{\ul d}:=p_1^*\mathcal O_{\C}(\underset{Z\in\T_{\ul{d}}(X^{pr})}{-\sum}Z),
\end{equation}
 where $p_1\:J^1_{f_d}\times_{B_d} \Y_d\ra \Y_d\ra \C$ is the composition of the second projection and the base change morphism, and $X^{pr}$ is the principal component of $C$. Let: 
$$p_2\:J^d_{f_d}\times_{B_d} J^1_{f_d}\times_{B_d} \Y_d\lra J^d_{f_d}\times_{B_d}\Y_d$$
$$p_3\:J^d_{f_d}\times_{B_d} J^1_{f_d}\times_{B_d} \Y_d\lra J^1_{f_d}\times_{B_d} \Y_d$$
be the projections and consider the composition:
$$\theta^{\ul d}_{f_d} \: J^d_{f_d}\times_{B_d} J^1_{f_d} \lra J^{d+1}_{f_d}\lra J^{d+1}_f,$$
where the first morphism is induced by 
$p_2^*(\L_d)\otimes p_3^*(\L_1)\otimes p_3^*(T_{\ul d})$. Set: $$U:=J^d_{f_d}\times_{B_d} J^1_{f_d} \text{ and } V:= J^d_f\times_B J^1_f.$$  
Let $q_1\:U\times_V U\ra U$ and $q_2\:U\times_V U\ra U$ be the two projections. 
Notice that $T_{\ul{d}}$ is the pull-back of a line bundle on $\C$. Then  $\theta^{\ul d}_{f_d}\circ q_1=\theta^{\ul d}_{f_d}\circ q_2$. Since $U$ is \'etale over $V$, by the flat descent there exists a morphism:
$$V=J^d_f\times_B J^1_f \stackrel{\theta^{\ul d}_f }{\lra} J^{d+1}_f$$
such that $\theta^{\ul{d}}_{f_d}$ factors via  $\theta^{\ul d}_f$.
Define $\ol{\alpha^2_f}$ as the composition:
$$\ol{\alpha^2_f}\:\C^2\stackrel {\ol{\alpha^1_f}\times\ol{\alpha^1_f}}{\lra}J^{1}_f \times_B J^{1}_f\stackrel{\theta^{\ul {e}_1}_f}{\lra}J^2_f,$$
where $\ul e_1$ is the multidegree defined in Proposition \ref{firstdeg}.
Due to the decomposition (\ref{Jacobdecomp}), there exists a unique multidegree $\ul e_2$ of total degree 2 such that $\ol{\alpha^2_f}|_{C^2}$ factors through the immersion 
$J_C^{\ul e_2}\hookrightarrow J_C^2$. This allows us to define: 
$$\ol{\alpha^3_f}\:\C^3\stackrel {\ol{\alpha^2_f}\times\ol{\alpha^1_f}}{\lra}J^{2}_f \times_B J^{1}_f\stackrel{\theta^{\ul {e}_2}_f}{\lra}J^3_f.$$
Arguing as before, $\ol{\alpha^3_f}|_{C^3}$ factors through the immersion $J_C^{\ul e_3}\hookrightarrow J_C^3$ for a unique multidegree $\ul e_3$ of total degree 3. Iterating, for every $d\geq2$,  we get a map:
$$
\ol{\alpha^d_f}\:\C^d\stackrel {\ol{\alpha^{d-1}_f}\times\ol{\alpha^1_f}}{\lra}J^{d-1}_f \times_B J^1_f\stackrel{\theta^{\ul{e}_{d-1}}_f}{\lra}J^d_f.
$$
and a unique multidegree $\ul e_d$ of total degree $d$ such that $\ol{\alpha^d_f}|_{C^d}$ factors through the immersion $J_C^{\ul e_d}\hookrightarrow J_C^d$.

\begin{Def}\label{Main-Def}
We call $\ol{\alpha^d_f}$ the \emph{$d$-th Abel maps} of the family $f$, for every $d\ge 1$. 
\end{Def}

Notice that the definition of $\ol{\alpha^d_f}$ only depends on the choice of the principal component of $C$, hence, from Lemma \ref{centracomp}, it is canonical if and only if $[C]\notin \Delta_{g/2}$. 
If $[C]\in \Delta_{g/2}$, then Lemma \ref{centracomp} implies that we get two $d$-th Abel maps. We will discuss the details of the special case in Section \ref{special}.

\smallskip

If $f\:\C\ra B$ is a smoothing of a curve of compact type and $D$ is a divisor 
of $\C$, then we set $\mathcal O_C(D):=\mathcal O_\C(D)\otimes\mathcal O_C$. Of course, being $C$ of compact type, the line bundle $\mathcal O_C(D)$ does not depend on the choice of $f$.

\begin{Lem}\label{twist}
Fix an integer $g\ge 2$. Let $C$ be a stable curve of compact type of genus $g$. Let $X$ be an irreducible component of $C$. Let $M$ be a 
$X$-quasistable line bundle on $C$ of multidegree $\ul{d}$. 
If $M'$ is a line bundle having degree 1 on $X$ and degree 0 on the other components of $C$, 
then the line bundle: 
$$M\ot M'\otimes\O_C(\underset{Z\in\T_{\ul{d}}(X)}{-\sum} Z)$$
is $X$-quasistable.
\end{Lem}

\begin{proof}
Let $\ul{d}'$ be the multidegree of $M\ot M'$ and $\ul{d}''$ the multidegree of: $$M\ot M'\otimes\O_C(\underset{Z\in\T_{\ul{d}}(X)}{-\sum} Z).$$ Let $d$, $d'$ and $d''$ be respectively the total degrees of $\ul{d}$, $\ul{d}'$ and  $\ul{d}''$.

\smallskip

\emph{First step}. 
Let $Z$ be a tail of $C$ not containing $X$. We claim that $Z$ is $\ul{d}$-big if and only if: 
\begin{equation}\label{cond1}
d'_Z-d'\frac{\deg(\omega_C|_Z)}{2g-2}<-\frac{1}{2}.
\end{equation}
Indeed, using $d'_Z=d_Z$, $d'=d+1$ and $\deg(\omega_C|_Z)=2g_Z-1$, we see that (\ref{cond1}) is equivalent to: 
$$d_Z\cdot \deg(\omega_C)-d \cdot \deg(\omega_C|_Z)<2g_Z- g.$$

\emph{Second step}.
We claim that $\ul{d}''$ is semistable 
at any connected subcurve $Y\subsetneq C$ such that 
$X\subseteq Y$. 
Indeed, let $Z$ be any connected component of $Y'$. Then $Z$ does not contain $X$. Furthermore, $Z$ is a tail, because $Y$ is connected. Since $\ul{d}$ is semistable at $Z$ and $d'=d+1$ and $d'_Z=d_Z$, we have: 
$$
-\frac{3}{2}=-\frac{k_Z}{2}-1\le d'_Z-d'\frac{\deg(\omega_C|_Z)}{2g-2}\le\frac{k_Z}{2}=\frac{1}{2}.
$$

Thus, using (\ref{cond1}), if $Z\in\T_{\ul{d}}(X)$, we have:
\begin{equation}\label{cond2}
-\frac{3}{2}\le d'_Z-d'\frac{\deg(\omega_C|_Z)}{2g-2}<-\frac{1}{2}.
\end{equation}
while if $Z\notin\T_{\ul{d}}(X)$, we have:
\begin{equation}\label{cond3}
-\frac{1}{2}\le d'_Z-d'\frac{\deg(\omega_C|_Z)}{2g-2}\le\frac{1}{2}.
\end{equation}

Suppose that $Z\in\T_{\ul{d}}(X)$. Then, being $Z'\notin \T_{\ul{d}}(X)$, we have $d''_Z=d'_Z+k_Z=d'_Z+1$. Hence by (\ref{cond2}) and $d''=d'$:
\begin{equation}\label{cond4}
-\frac{1}{2}\le d''_Z-d''\frac{\deg(\omega_C|_Z)}{2g-2}<\frac{1}{2}.
\end{equation}
and $\ul{d}''$ is semistable at $Z$.
If $Z\notin\T_{\ul{d}}(X)$, then $d''_Z=d'_Z$, and hence, by (\ref{cond3}), 
$\ul{d}''$ is semistable at $Z$. In particular, $\ul{d}''$ is semistable at $Y'$, hence it is semistable at $Y$.

\smallskip

\emph{Third step}. 
We show that $\ul{d}''$ is semistable
at any non-empty connected subcurve $Y\subsetneq C$ not containing $X$. Indeed, if $\ul{d}''$ is not semistable at any such $Y$, then it is not semistable at $Y'$. Write $Y'=Y_1\cup Y_2$, where $Y_1$ is connected, 
$Y_1\cap Y_2=\emptyset$ and $X\subseteq Y_1$.  By the second step, $\ul{d}''$ is semistable at $Y_1$. Then $Y_2\ne\emptyset$, because otherwise we would have $Y'=Y_1$ and $\ul{d}''$ would be semistable at $Y'$. Notice that $Y'_2=Y\cup Y_1$, hence $Y'_2$ is connected and $X\subset Y'_2$. Again by the third step, $\ul{d}''$ is semistable at $Y'_2$.

If $\ul{d}''$ is semistable at $Y_2$, then $\ul{d}''$ is semistable at 
$Y_1\cup Y_2=Y'$, because $Y_1\cap Y_2=\emptyset$, which is a contradiction. Thus, $\ul{d}''$ is not semistable at $Y_2$, hence it is not semistable at $Y'_2$, which is again a contradiction.

\smallskip

\emph{Fourth step}. We show that $\ul{d}''$ is $X$-quasistable.  By the third and fourth steps, $\ul{d}''$ is semistable 
at any non-empty connected subcurve $Y\subsetneq C$, hence $\ul{d}''$ is semistable. 
Thus we are done if we show that $\ul{d}''$ is $X$-quasistable at any non-empty  subcurve $Y\subsetneq C$ such that $X\subseteq Y$.
Let $Y$ be any such subcurve. We can assume without loss of generality that $Y$ is connected. We distinguish two cases.
If there exists a connected component $Z$ of $Y'$ such that  $Z\in\T_{\ul{d}}(X)$, then from (\ref{cond3}) and (\ref{cond4}) we have:
$$d''_{Y'}-d''\frac{\deg(\omega_C|_{Y'})}{2g-2}<\frac{k_Y}{2}.$$
Using that $d''_{Y'}=d''-d''_Y$ and 
$\deg(\omega_C|_{Y'})=\deg(\omega_C)-\deg(\omega_C|_Y)$,  
we see that $\ul{d}''$ is $X$-quasistable at $Y$. 

If $Z\notin\T_{\ul{d}}(X)$, for each connected component $Z$ of $Y'$, then $d''_Y=d'_Y$. Then, since $\ul{d}$ is $X$-quasistable at $Y$ and using that 
$d_Y=d'_Y-1$ and $d=d'-1$, we have: 
$$
d'_Y-d'\frac{\deg(\omega_C|_{Y})}{2g-2}>-\frac{k_Y}{2}.
$$ 
Since $d''=d'$ and $d''_Y=d'_Y$, we see that  $\ul{d}''$ is $X$-quasistable at $Y$.
\end{proof}

\begin{Thm}\label{abelcompact}
Fix integers $g\ge 2$ and $d\ge 1$. Let $f\:\C\ra B$ be a smoothing of a stable curve $C$ of compact type of genus $g$ and $X^{pr}$ be the principal component of $C$. Let $\ol{\alpha^d_f}|_{C^d}\:C^d \ra J_C^{\ul{e}_d}$ 
be the restriction of the $d$-th Abel map of $f$ to the special fiber.
Then the following properties hold:
\begin{itemize}
\item[(i)]
$\ul{e}_d$ is a $X^{pr}$-quasistable multidegree.
\item[(ii)]
$\ol{\alpha^d_f}|_{C^d}$
does not depend on the choice  of $f$ and  factors via a morphism  
$\beta^d_C\:S^d(C)\ra J_C^{\ul{e}_d}$. 
\end{itemize}
\end{Thm}

\begin{proof} 
Since $C$ is a proper scheme, then $J^d_C$ is a scheme and there exists a universal degree-$d$ line bundle $\L_d$ over $J^d_C\times C$. 
Call: $$\phi_d\:J^d_C\times C\ra J^d_C$$ the projection. 
We show (i) and that $\ol{\alpha^d_f}|_{C^d}$ does not depend on the choice of $f$, arguing by induction on $d$. If $d=1$, then we are done by Proposition \ref{firstdeg}. 
Consider:
$$\ol{\alpha^{d+1}_f}|_{C^{d+1}}\:C^{d+1}\stackrel {\ol{\alpha^d_f}|_{C^d}\times\ol{\alpha^1_f}|_C}{\lra}J^{\ul{e}_d}_C \times  J^{1,X^{pr}}_C\stackrel{\theta^{\ul{e}_d}_f}{\lra}J^{\ul{e}_{d+1}}_C\hookrightarrow J^{d+1}_C.$$
By induction, $\ol{\alpha^d_f}|_{C^d}$ does not depend on the choice of $f$, thus 
$\ol{\alpha^{d+1}_f}|_{C^{d+1}}$ does not depend on the choice of $f$ either.
Take a point $p$ in the image of $\ol{\alpha^{d+1}_f}|_{C^{d+1}}$.  Let $L_p$ be the restriction of $\L_{d+1}$ to the fiber $C_p=\phi_{d+1}^{-1}(p)$. By construction, we have:
$$L_p=M_p\otimes M'_p\otimes M''_p,$$ where $[M_p]\in J^{\ul{e}_d}_C$ and $[M'_p]\in J^{1,X^{pr}}_C$, and by (\ref{modif}): 
$$M''_p\simeq \O_{\C}(\underset{Z\in\T_{\ul{e}_d}(X^{pr})}{-\sum} Z)|_C.$$ 
By induction, 
$\ul e_d$ is $X^{pr}$-quasistable. Hence,  
by Lemma \ref{twist}, $L_p$ is a $X^{pr}$-quasistable line bundle 
and $\ul{e}_{d+1}$ is a $X^{pr}$-quasistable multidegree. 

To complete the proof, we show  that $\ol{\alpha^d_f}|_{C^d}$ is invariant under the action of the $d$-th symmetric group on $C^d$. Indeed, for every $d\ge 1$, pick the following line bundle of $J^1_C\times C$:
$$P_d:=\underset{1\le i\le d-1}\bigotimes  \O_{J^1_C\times C}(\underset{Z\in\T_{\ul{e}_i}(X^{pr})}{-\sum} J^1_C\times Z).$$ 
For every $i\ge 1$, consider the projection:
$$q_i\:J^{1,X^{pr}}_C\times\cdots\times J^{1,X^{pr}}_C\times C\ra J^{1,X^{pr}}_C\times C$$
onto the product of the $i$-th factor and $C$. Then $\ol{\alpha^d_f}|_{C^d}$ factors as:
$$\ol{\alpha^d_f}|_{C^d}\:C^d\stackrel{\ol{\alpha^1_f}|_C\times\cdots\times\ol{\alpha^1_f}|_C}{\lra}J^{1,X^{pr}}_C\times\cdots\times J^{1,X^{pr}}_C\stackrel{\rho_d}{\lra}J^{\ul{e}_d}_C$$ 
where $\rho_d$ is the morphism induced by $q_1^*\L_1\otimes\cdots \otimes q_d^*\L_1\otimes q_d^*P_d$. We see that 
$\ol{\alpha^d_f}|_{C^d}$ is invariant under the action of the $d$-th symmetric group on $C^d$.
\end{proof}

\begin{Def}
Keep the notations of Theorem \ref{abelcompact}. We set $\alpha^d_C:=\ol{\alpha^d_f}|_{C^d}$. We call $\alpha^d_C$ the \emph{$d$-th Abel map of $C$}, and $\beta^d_C$ the \emph{symmetric $d$-th Abel map of $C$}.
\end{Def}

\begin{Exa}
Fix an integer $g\ge 2$. Let $C$ be a stable curve of compact type of genus $g$  with two components $C_1$ and $C_2$ such that $g_{C_1}\ge g_{C_2}$.  Let $X^{pr}$ be the principal component of $C$ and $n$ be the node of $C$. Set $n_1:=C_1\cap C_2\subseteq C_1$ and $n_2:=C_1\cap C_2\subseteq C_2$. We may assume without loss of generality that $X^{pr}=C_1$. Then $\ul{e}_1=(1,0)$ and: 
$$\beta^1_C(n)\otimes\mathcal O_{C_1}\simeq\mathcal O_{C_1}(n_1) \text{ and }\beta^1_C(n)\otimes\mathcal O_{C_2}\simeq\mathcal O_{C_2}.$$ 
It is easy to see that: 
$$
\T_{\ul{e}_1}(X^{pr})=
\begin{cases}
\begin{array}{ll}
\{C_2\} & \text{ if } 4g_{C_2}>g+1 \\
\emptyset & \text{ if } 4g_{C_2}\le g+1 
\end{array}
\end{cases}
$$
and hence:
$$\beta^2_C(n,n)\otimes\mathcal O_{C_1}\simeq
\begin{cases}
\begin{array}{ll}
 \mathcal O_{C_1}(n_1) & \text{ if } 4g_{C_2}>g+1 \\
\mathcal O_{C_1}(2n_1) & \text{ if } 4g_{C_2}\le g+1
\end{array}
\end{cases}
$$ 
$$\beta^2_C(n,n)\otimes\mathcal O_{C_2}\simeq
\begin{cases}
\begin{array}{ll}
\mathcal O_{C_2}(n_2) & \text{ if } 4g_{C_2}>g+1 \\
\mathcal O_{C_2} & \text{ if } 4g_{C_2}\le g+1 
\end{array}
\end{cases}
$$
Thus $\ul{e}_2=(1,1)$ if $4g_{C_2}>g+1$ and $\ul{e}_2=(2,0)$ if $4g_{C_2}\le g+1$. 
\end{Exa}

In Proposition \ref{hyper-char}, we will give a characterization of hyperelliptic curves of compact type with two components, via the symmetric $2$-nd Abel map. 
Recall that a stable curve $C$ of genus $g\ge 2$ is hyperelliptic if the closure of the locus in $\ol{M_g}$ of hyperelliptic smooth curves of genus $g$ contains $[C]$.

\begin{Prop}\label{hyper-char}
Fix an integer $g\ge 2$. Let $C$ be a stable curve of compact type of genus $g$  with two components. Then $C$ is hyperelliptic if and only if there exists a fiber of the symmetric $2$-nd Abel map $\beta^2_C$ of $C$ consisting of two smooth rational curves intersecting in one point.
\end{Prop}

\begin{proof}
Let $C_1$ and $C_2$ be the components of $C$. Let $n$ be the node of $C$, and set $n_1:=C_1\cap C_2\in C_1$ and $n_2:=C_2\cap C_2\in C_2$. We may assume without loss of generality that $C_1$ is the principal component of $C$.  It is well-known that $C$ is hyperelliptic if and only if $C_i$ is hyperelliptic and $|2n_i|$ is the $g^1_2$ of $C_i$, for $i=1,2$.  By construction, for every $p_i\in C_i$, and $i,j\in\{1,2\}$:
$$\beta^1_C(p_i)\otimes \O_{C_j}\simeq
\begin{cases}
\begin{array}{ll}
\mathcal O_{C_j}(p_i-(i-1)n_j) & i=j
\\
\O_{C_j}((i-1)n_j) & i\ne j
\end{array}
\end{cases}
$$
Assume that $\ul{e}_2=(1,1)$. The case $\ul{e}_2=(2,0)$ is completely analogous. By the definition of the symmetric $2$-nd Abel map we have for every $i,j,k\in\{1,2\}$ and $p_i\in C_i$, $q_j\in C_j$, :
$$
\beta^2_C(p_i,q_j)|_{C_k}\simeq
\begin{cases}
\begin{array}{ll}
\mathcal O_{C_k}(p_k+q_k-n_k) & \text{ if } i=j=k
\\
\mathcal O_{C_k}(n_k) & \text{ if } i=j\ne k
\\
\mathcal O_{C_k}(p_k) & \text{ if } k=i\ne j
\\
\mathcal O_{C_k}(q_k) & \text{ if } i\ne j=k
\end{array}
\end{cases}
$$

Thus $\beta^2_C(p,q)=\beta^2_C(p',q')$ for $(p,q)\ne (p',q')$ if and only if at least one of the following cases holds:
\begin{itemize}
\item[(a)]
$p,q,p',q'\in C_1$ with $|p+q|=|p'+q'|$; in particular, $C_1$ is hyperelliptic.
\item[(b)]
$p,q,p',q'\in C_2$ with $|p+q|=|p'+q'|$; in particular, $C_2$ is hyperelliptic.
\item[(c)]
$p,q\in C_1$ and $p',q'\in C_2$ with $|p+q|=|2n_1|$ and $|p'+q'|=|2n_2|$; in particular, $C_1,C_2$ are hyperelliptic and $|2n_1|$, $|2n_2|$ are the $g^1_2$'s.
\end{itemize}

Denote by $F_{p,q}:=(\beta^2_C)^{-1}(\beta^2_C(p,q)).$ 
Notice that the dimension of $F_{p,q}$ is at most 1. If the dimension of $F_{p,q}$ is 1, then (a) and (b) imply that $F_{p,q}$ has at most two components $E_1$ and $E_2$, given by:
\begin{equation}\label{E1}
E_i=\{(p,q)\in S^2(C_i) : |p+q| \tx{ is the } g^1_2 \text{ of } C_i\}\simeq \mathbb{P}^1, \text{ for } i=1,2.
\end{equation}

Assume that $C$ is not hyperelliptic, and that $F_{p,q}$ is a curve with two components, as in (\ref{E1}). 
Notice that (c) does not hold, because $C$ is not hyperelliptic. Then $E_1\subseteq S^2(C_1)-(n,n)$ and $E_2\subseteq S^2(C_2)-(n,n)$, and hence $E_1\cap E_2=\emptyset$.

Conversely, if $C$ is hyperelliptic, then (a), (b) and (c) hold. In particular, $F_{n,n}=E_1\cup E_2$, where $E_1, E_2$ are as in (\ref{E1}) and  
 $E_1\cap E_2=(n,n)$.
\end{proof}

\subsection{Abel maps with other targets}\label{special}

Let $f\:\C\to B$ be a family of stable curves.
Fix an integer $d\ge 1$. We denote by  
$J^{d, ss}_f$ the relative version of  $J^{d,ss}_C$, i.e. the $B$-scheme whose fiber over $b$ is $J^{d,ss}_{C_b}$.

 Let $P^d_f$ be the relative generalized Jacobian of the family $f$. A geometrically meaningful compactification $\ol{P^d_f}\ra B$ of $P^d_f$ is constructed in \cite{C}. It is the same compactification of $P^d_f$ produced in \cite{S} and \cite{Pand}.  The fiber $\ol{P^d_{C_b}}$ over $b\in B$ parametrizes the equivalence classes (under a suitable equivalence relation) of pairs $(X, L)$, where $X$ is a nodal curve obtained by blowing up $C_b$ and $L$ is a degree-$d$ line bundle on $X$, whose multidegree satisfies the numerical condition (\ref{BI1}). In particular, we get a set-theoretic map:
$$
\psi_d\:J^{d,ss}_f\lra \ol{P^d_f}
$$
which is indeed a morphism.  A fiber of $\psi_d$ parametrizes the set of line bundles contained in a Jordan-H\"older equivalence class of rank-1 torsion free semistable sheaves. We will refer to \cite[Section 8]{E01} for more details.

If $C$ is of compact type and $\alpha^d_C$ is its $d$-th Abel map, we obtain an other Abel map $\psi_d\circ\alpha^d_C$ whose target is $\ol{P^d_C}$. 
Some natural questions arise:
\begin{itemize}
\item[(Q1)]
do $\alpha^d_C$ and $\psi_d\circ\alpha^d_C$ have the same set-theoretic fibers?
\item[(Q2)]
if $[C]\in \Delta_{g/2}$, and $\alpha^d_1$ and $\alpha^d_2$ are the two Abel maps of $C$ whose target is $J^{d,ss}_C$, do $\psi_d\circ\alpha^d_1$ and $\psi_d\circ\alpha^d_2$ have the same set-theoretic fibers?
\end{itemize}

We will answer positively to the posed questions.

\begin{Prop}\label{same-fiber}
Fix an integer $g\ge 2$. Let $C$ be a curve of compact type of genus $g$. For every integer $d\ge 1$, consider the $d$-th Abel map  $\alpha^d_C\: C^d\ra J^{d,ss}_C$ of $C$ and  the morphism $\psi_d\:J^{d,ss}_C\ra\ol{P^d_C}$. Then $\alpha^d_C$ and $\psi_d\circ\alpha^d_C$ have the same set-theoretic fibers.
\end{Prop}

\begin{proof}
We will show that $\psi_d$ is injective over $J^{\ul{d}}_C$, for every semistable multidegree $\ul{d}$ of total degree $d$. Suppose that $L$ and $M$ are line bundles of $C$ with multidegree $\ul{d}$ and such that $\psi_d([L])=\psi_d([M])$. Then, 
from \cite[Theorem 5.1.6]{CCC} there exists a curve $W$ obtained by blowing up $C$ and a smoothing $\W\ra B$ of $W$ such that, if we denote by $\pi\:W\ra C$ the morphism of blow up and by $Y_1,\dots Y_\gamma$ the irreducible components of $W$, then: 
$$\pi^*L\otimes\pi^*M^{-1}\simeq\mathcal O_{\W}(\sum_{i=1}^\gamma a_{Y_i}\cdot Y_i)|_W \, , \, \tx{ for some } (a_{Y_1},\dots a_{Y_\gamma})\in \ze^\gamma.$$ Since $M$ and $L$ have the same multidegree, it follows that $a_{Y_i}=0$ for each $i=1,\dots,\gamma$. In particular $L|_X\simeq M|_X$ for each component $X$ of $C$. Since $C$ is of compact type, from (\ref{iso-compo}) we have that $L\simeq M$ and we are done.
\end{proof}

\begin{Rem}\label{higher-maps}
It follows from \cite{CE} and \cite{CCE} that one can answer positively to the analogous of question (Q1) for the first Abel map of a stable curve.
 Nevertheless, we believe that this phenomenon does not take place for higher Abel maps of stable curves not of compact type, as the following example hints.

Fix a integer $g\ge 2$.  Let $C$ be a stable curve of genus $g$ with components $X_1,\dots,X_\gamma$. Fix an integer $d\ge 2$. Let $\dot{C}^d\subset C^d$ be the subset of  the $d$-tuples of smooth points of $C$. For every semistable multidegree $\ul{d}$ of total degree $d$, consider the following subset of 
$C^d$: 
$$\dot{C}^{\ul{d}}:=\{(p_1,\dots,p_d)\in \dot{C}^d : \#(X_i\cap \{p_1,\dots,p_d\})=d_{X_i}, \tx {for } i=1,\dots,\gamma\}.$$ 

There exists  a natural Abel map defined on $\dot{C}^{\ul{d}}$. In fact, consider  the line bundle $\I:=\mathcal O_{\dot{C}^{\ul{d}}\times C}(\sum_{r=1}^d\Delta_r)$ on  $\dot{C}^{\ul{d}}\times C$, where:
$$\Delta_r:=\{(p_1,\dots,p_d,p)\in \dot{C}^{\ul{d}}\times C : p_r=p\}, \tx{ for } r=1,\dots, d.$$
Pick the trivial family of curves
$\pi\:\dot{C}^{\ul{d}}\times C\ra \dot{C}^{\ul{d}}$, where $\pi$ is the projection.  Then $\I$ yields a family of semistable line bundles of $C$ 
over the base  
$\dot{C}^{\ul{d}}$. Since $J^{d,ss}_C$ is a fine moduli scheme, 
 we get a morphism $\alpha^{\ul{d}}_C \: \dot{C}^{\ul{d}}\ra J^{d,ss}_C$ such that: 
 $$\alpha^{\ul{d}}_C(p_1,\dots,p_d)=\mathcal O_{C}(\sum_{i=1}^d p_i).$$

Suppose now that $C$ is a binary curve of genus $g\ge 2$, i.e. a stable curve with two smooth rational components intersecting at $g+1$ nodes. Let $X_1$ and $X_2$ be the components of $C$ and let $q_1,\dots,q_{g+1}$ be the nodes of $C$. 
Consider the following subset of $J^{g+1}_C$:
$$
\mathcal F:=\{L\in J^{g+1}_C : L|_{X_1}\simeq\mathcal O_{X_1}(\sum_{i=1}^{g+1}q_i)
\, ; \, L|_{X_2}\simeq\mathcal O_{X_2}\}/\ \tx{iso}
$$ 
As a set, $\mathcal F=(k^*)^g$. It is easy to check that if $L\in\mathcal F$, then $L$ is a semistable line bundle. For each $L\in\mathcal F$, we have $h^0(C,L)=2+h^0(C,\omega_C\otimes L^{-1})$, 
by Riemann-Roch, and:
$$H^0(C,\omega_C\otimes L^{-1})\subseteq H^0(C,\omega_C|_{X_1}\otimes L^{-1}|_{X_1})\oplus H^0(C,\omega_C|_{X_2}\otimes L^{-1}|_{X_2})\simeq $$
$$\simeq H^0(\mathbb{P}^1,\mathcal O_{\mathbb{P}^1}(-3-g))\oplus H^0(\mathbb{P}^1, \mathcal O_{\mathbb{P}^1}(-2))=0.$$
Hence $h^0(C, L)=2$. Since $L|_{X_2}\simeq\mathcal O_{X_2}$, we have that 
if $s(q_i)=0$ for some $s\in H^0(C, L)$ and $i=1,\dots,g+1$, then $s|_{X_2}=0$. 
Hence, for every $i=1,\dots, g+1$: 
$$H^0(C, L(-q_i))\simeq H^0(C, L(-\sum_{i=1}^{g+1}q_i))\simeq H^0(X_1, L|_{X_1}(-\sum_{i=1}^{g+1}q_i))\simeq k.$$
Hence we have $H^0(C, L)\simeq s_L\cdot k\oplus H^0(C, L(-\sum_{i=1}^{g+1}q_i)$, where $s_L$ is a section such that $s_L(q_i)\ne 0$, for each $i=1,\dots,g+1$. 
Consider the set: 
$$\hat{\mathcal F}:=\{(p_1,\dots,p_{g+1}) : \tx{div}(s_L)=p_1+\dots p_{g+1}\, , \, \tx{ for some } L\in\mathcal F\}\subseteq \dot{C}^{(g+1, 0)}.$$
In particular, we have a bijection: 
\begin{equation}\label{bij}
\alpha^{(g+1,0)}_C|_{\hat{\mathcal F}}\:\hat{\mathcal F}\stackrel{1:1}{\lra} \mathcal F.
\end{equation}

Consider now the morphism $\psi_{g+1}\:J^{g+1,ss}_C\ra \ol{P^{g+1}_C}$. We claim that $\psi_{g+1}$ contracts 
$\mathcal F$ to a point. Indeed, recall the definition (\ref{polar}) of 
 canonical polarization $E_{g+1}$. Then:
\begin{itemize}
\item[(a)]
if we denote $\ul{d}_L:=(g+1,0)$, the multidegree of $L$,
 for each $L\in\mathcal F$, then: $$\chi(L|_{X_2}\otimes E_{g+1}|_{X_2})=(d_L)_{X_2}-\frac{g+1}{2g-2}\cdot\deg(\omega_C|_{X_2})+\frac{k_{X_2}}{2}=0$$
and in particular $L$ is not $X_2$-quasistable;
\item[(b)]
by construction, we have $L|_{X_2}\simeq L'|_{X_2}\simeq\mathcal O_{X_2}$ and $\ker(L\ra L|_{X_2})\simeq \ker(L'\ra L'|_{X_2})\simeq\mathcal O_{X_1}$, for each $L,L'\in\mathcal F$.
\end{itemize}
Following \cite[Section 1.3]{E01}, the properties (a) and (b) imply that $L$ and $L'$ are Jordan-H\"older equivalent, for each $L,L'\in\mathcal F$. In particular, by \cite[Section 8]{E01}, the morphism $\psi_{g+1}$ contracts $\mathcal F$ to a point. Hence, recalling the bijection (\ref{bij}), we get a sequence of morphisms:
$$\psi_{g+1}\circ\alpha^{(g+1,0)}_C|_{\hat{\mathcal F}}\:
\hat{\mathcal F}\stackrel{1:1}{\lra}\mathcal F\stackrel{\psi_{g+1}}{\lra}\{pt\},$$
showing that the fibers of $\alpha^{(g+1,0)}_C$ and $\psi_{g+1}\circ\alpha^{(g+1,0)}_C$ are different.
\end{Rem}

\begin{Prop}
Fix an even integer $g\ge 2$. Let $C$ be a curve of compact type of genus $g$ such that $[C]\in \Delta_{g/2}$. Let $X_1$ and $X_2$ be the semicentral components of $C$. 
For each integer $d\ge 1$, let $\alpha^d_1$ (respectively $\alpha^d_2$) be the Abel map of $C$, once we choose $X_1$ (respectively $X_2$) to be the principal component of $C$. 
If $Y$ is the tail of $X'_1$ such that $g_Y=g/2$, then there exists an integer $\eta_d\in\{-1, 0, 1\}$ such that: 
\begin{equation}\label{twist-relation}
\alpha^d_1(p)\simeq \alpha^d_2(p)\otimes \mathcal O_C(\eta_d \cdot Y) \tx{ for each } p\in C^d.
\end{equation}
In particular, if we consider the morphism $\psi_d\:J^{d,ss}_C\ra\ol{P^d_C}$, then $\psi_d\circ\alpha^d_1$ and $\psi_d\circ\alpha^d_2$ have the same set-theoretic fibers.
\end{Prop}

\begin{proof}
Recall that $X_1\cap X_2\ne\emptyset$, from Lemma \ref{centracomp} (iii).
Let $\{\ul{e}_{1,1},\dots,\ul{e}_{1,d},\dots\}$ (resp. $\{\ul{e}_{2,1},\dots,\ul{e}_{2,d},\dots\}$) be the set of multidegrees induced by 
the Abel maps $\alpha^1_1,\dots,\alpha^d_1\dots$ (resp. $\alpha^1_2,\dots,\alpha^d_2\dots$).
We show (\ref{twist-relation}) by induction on $d$. Indeed, it is true if $d=1$ with $\eta_1=1$, as explained in \cite{CCE}. Fix an integer $d\ge 2$. For each tail $Z$ of $C$, and for $i=1,2$, set:
$$\epsilon_{i, d, Z}:=
\begin{cases}
\begin{array}{ll}
0 & \tx{ if } Z\notin \T_{\ul{e}_{i, d}}(X_i)\\
1 & \tx{ if } Z\in \T_{\ul{e}_{i,d}}(X_i)
\end{array}
\end{cases}$$
Let $Y_1$ be the tail of $X'_2$ such that $g_{Y_1}=g/2$ and set $Y_2:=Y$. 
For each $p=(p_1,\dots,p_d,p_{d+1})\in C^{d+1}$, set $\hat{p}:=(p_1,\dots,p_d)\in C^d$. 
We have:
$$\alpha^1_1(p_{d+1}) \simeq\alpha^1_2(p_{d+1})\otimes\mathcal O_C(Y_2)$$
and by induction:
$$\alpha^d_1(\hat{p})\simeq\alpha^d_2(\hat{p})\otimes\mathcal O_C(\eta_d\cdot Y_2).$$
In particular $(e_{1,d})_Y=(e_{2,d})_Y$ for each subcurve $Y$ such that 
$Y\subseteq (X_1\cup X_2)'$. Then we have: 
$$\alpha^{d+1}_1(p)\simeq\alpha^d_1(\hat{p})\otimes\alpha^1_1(p_{d+1})\otimes\mathcal O_{C}(-\sum_{Z\in\T_{\ul{e}_{1, d}}(X_1)}Z)\simeq$$
$$\simeq\alpha^d_1(\hat{p})\otimes\alpha^1_1(p_{d+1})\otimes\mathcal O_C(-\epsilon_{1, d, Y_2}\cdot Y_2-\sum_{Z\in\T_{\ul{e}_{1, d}}(X_1\cup X_2)}Z)\simeq$$
$$\simeq \alpha^d_2(\hat{p})\otimes \alpha^1_2(p_{d+1})\otimes\mathcal O_C((\eta_d+1-\epsilon_{1, d, Y_2})\cdot Y_2-\sum_{Z\in\T_{\ul{e}_{1, d}}(X_1\cup X_2)}Z)\simeq$$
$$\simeq\alpha^{d+1}_2(p)\otimes\mathcal O_C(\epsilon_{2, d, Y_1}\cdot Y_1+(\eta_d+1-\epsilon_{1, d, Y_2})\cdot Y_2)\simeq$$
$$\simeq\alpha^{d+1}_2(p)\otimes\mathcal O_C(\eta_{d+1}\cdot Y_2)$$
where 
$$
\eta_{d+1}:=\eta_d+1-\epsilon_{2, d, Y_1}-\epsilon_{1, d, Y_2}.
$$
We show that $|\eta_{d+1}|\le 1$. Indeed, if 
$|\eta_{d+1}|>2$, then $|(e_{1, d+1})_{Y_1}-(e_{2, d+1})_{Y_1}|>2$.  
Being $\ul{e}_{1,d+1}$ semistable at $Y_1$, we have: 
$$-\frac{1}{2}\le (e_{1,d+1})_{Y_1}-\frac{d}{2g-2}\deg\omega_C|_{Y_1}\le\frac{1}{2}$$
and hence $\ul{e}_{2, d+1}$ is not semistable at $Y_1$, which is a contradiction.

To show the last statement, we show that 
$\psi_d(\alpha^d_1(p))=\psi_d(\alpha^d_2(p))$, for each $p\in C^d$. Let $f\:\C\ra B$ be any smoothing of $C$. Pick a coeherent sheaf $\mathcal I$ on $\C$, flat over $B$, such that $\mathcal I|_{C_b}$ is a line bundle on $C_b$ for each $b\in B$ and $\I|_C\simeq \alpha^d_1(p)$. Consider $\I':=\I\otimes\mathcal O_{\C}(\eta_d\cdot Y)$. 
Since $C$ is of compact type, we have 
$\I'|_C\simeq \alpha^d_2(p)$. By \cite[Proposition 8.1]{C}, we get morphisms: 
$$\psi_{\I}\:B\ra \ol{P^d_f} \tx{ \, and \, } \psi_{\I'}\:B\ra \ol{P^d_f} $$ 
such that $\psi_{\I}(b)=\I|_{C_b}$ and $\psi_{\I'}(b)=\I'|_{C_b}$ for each $b\in B$. In particular, $\psi_\I(b)=\psi_{\I'}(b)$ for each $b\ne 0$, and $\psi_\I(0)=\psi_d(\alpha^d_1(p))$ and $\psi_{\I'}(0)= \psi_d(\alpha^d_2(p))$. Since $\ol{P^d_f}$ is a separated scheme, we get  $\psi_d(\alpha^d_1(p))=\psi_d(\alpha^d_2(p))$.
\end{proof}

\section*{Acknowledgments}
It is a pleasure to thank Eduardo Esteves for fundamental suggestions and very useful discussions, and for carefully reading a preliminar version of the paper. We thank also Lucia Caporaso for precious discussions.

\end{document}